\documentclass[a4paper,oneside,11pt]{amsart}

\reversemarginpar

\usepackage[colorlinks,hyperindex,linkcolor=blue,urlcolor=black,pdftitle={Some relationships with subnormal operators and existence of hyperinvariant subspaces},pdfauthor={Maria F.Gamal'}]
{hyperref}

\usepackage{amsmath}
\usepackage{amsfonts}
\usepackage{amssymb}
\usepackage{amsthm}

\usepackage[english]{babel}
\usepackage{enumerate}

\usepackage{graphicx}
\usepackage{color}

\usepackage[abbrev]{amsrefs}

\numberwithin{equation}{section}

\theoremstyle{plain}
\newtheorem{theorem}{Theorem}[section]
\newtheorem{lemma}[theorem]{Lemma}
\newtheorem{corollary}[theorem]{Corollary}
\newtheorem{proposition}[theorem]{Proposition}

\theoremstyle{definition}
\newtheorem{remark}[theorem]{Remark}
\newtheorem{example}[theorem]{Example}

\begin{document}

\title[Existence of hyperinvariant subspaces]{Some relationships with subnormal operators and existence of hyperinvariant subspaces}

\author{Maria F. Gamal'}
\address{
 St. Petersburg Branch\\ V. A. Steklov Institute 
of Mathematics\\
 Russian Academy of Sciences\\ Fontanka 27, St. Petersburg\\ 
191023, Russia  
}
\email{gamal@pdmi.ras.ru}

\renewcommand{\thefootnote}{}

\footnote{2020 \emph{Mathematics Subject Classification}: Primary 47A15; Secondary 47B02, 47A60, 47B20.}

\footnote{\emph{Key words and phrases}: hyperinvariant subspace, polynomially bounded operator, subnormal operator, unilateral shift.}




\begin{abstract}
If $T$ is a polynomially bounded operator, $\mathcal M$ is an invariant subspace of $T$, 
$T|_{\mathcal M}$ is a unilateral shift and $T^*|_{\mathcal M^\perp}$ is subnormal, then $T$ has a nontrivial hyperinvariant subspace. 
If an operator $T$ is intertwined from both sides with two operators, one of which is hyponormal and other is the adjoint to hyponormal,
 then $T$ has a nontrivial hyperinvariant subspace. The existence of nontrivial hyperinvariant subspaces for subnormal operators themselves is not studied here. 
\end{abstract}

\maketitle

\section{Introduction}

Let $\mathcal H$ be a (complex, separable) Hilbert space, and let  $\mathcal L(\mathcal H)$ be the algebra of all (bounded linear) operators acting on  $\mathcal H$.
 The algebra of all $R\in\mathcal L(\mathcal H)$  such that $TR=RT$ is called the \emph{commutant} of $T$ and is denoted by  $\{T\}'$.  A (closed) subspace  $\mathcal M$ of  $\mathcal H$ is called \emph{invariant} 
for an operator  $T\in\mathcal L(\mathcal H)$, if $T\mathcal M\subset\mathcal M$, and \emph{hyperinvariant} 
for $T$ if $R\mathcal M\subset\mathcal M$ for all $R\in\{T\}'$. The complete lattice of all invariant (resp.,  hyperinvariant) subspaces for  $T$ is denoted by  $\operatorname{Lat}T$ (resp., by 
$\operatorname{Hlat}T$). 
 The \emph{hyperinvariant subspace problem}  is the question whether for every nontrivial operator 
$T\in\mathcal L(\mathcal H)$ there exists a nontrivial hyperinvariant subspace. 
Here ``nontrivial operator'' means  not a scalar multiple of the identity operator, 
and ``nontrivial subspace'' means  different from  $\{0\}$ and  $\mathcal H$.

Recall that an operator $A\in\mathcal L(\mathcal H)$ is called \emph{subnormal}
 if there exists a complex Hilbert space $\mathcal K$ and a normal operator $N\in\mathcal L(\mathcal K)$ 
 such that $\mathcal H\subset\mathcal K$, $\mathcal H\in\operatorname{Lat}N$ and $A=N|_{\mathcal H}$. Every subnormal
operator $A$ has  a unique (up to unitary equivalence) minimal normal
extension, see {\cite[Corollary II.2.7]{conwaysubnormal}}.

Existence of invariant and hyperinvariant subspaces for operators relating to normal ones in various sense
 is considered, for example, in 
\cite{kimpearcy}, \cite{aglermccarthy}, \cite{jungkopearcy}. 
It is known that subnormal operator is reflexive {\cite[Theorem VII.8.5]{conwaysubnormal}},
 and rationally cyclic subnormal operator  has nontrivial hyperinvariant subspaces
(see {\cite[Corollary V.4.7]{conwaysubnormal}} or \cite{thomson}). For some other results, see \cite{foiasjungkopearcy}. 
But (up the author's knowledge) 
the existence of nontrivial hyperinvariant subspaces for arbitrary subnormal operator is unknown.

In the present paper, 
 the existence of a nontrivial hyperinvariant subspace is proved for 
operators
$T\in\mathcal L(\mathcal H)$ 
which admit an $H^\infty$-functional calculus and have 
$\mathcal H_1\in\operatorname{Lat}T$ such that $T|_{\mathcal H_1}$ is a unilateral shift and $T^*|_{\mathcal H_1^\perp}$ is 
a subnormal operator having a normal extension with spectral measure concentrated on the open unit disc. Under these assumptions, 
there exist singular inner  functions $\theta$ such that $\operatorname{clos}\theta(T)\mathcal M\neq\mathcal M$ for every 
$\mathcal M\in\operatorname{Lat}T$ such that $\mathcal H_1\subset\mathcal M$. Consequently, $T|_{\mathcal M}$ has a nontrivial invariant subspace $\operatorname{clos}\theta(T)\mathcal M$ for every such $\mathcal M$ (Theorem \ref{thmmain1} and Corollaries \ref{coruu} and \ref{cormm}). The proof is based on \cite{est}.  In Remark \ref{rem1} 
a comparison with some results from \cite{kimpearcy} and \cite{jungkopearcy} is given. 

Also, it is proved that if an operator $T$ is intertwined from both sides with two operators, one of which is hyponormal and other is the adjoint to hyponormal,
 then $T$ has a nontrivial hyperinvariant subspace (Theorem \ref{thmmain2}).
 This is a generalization of  {\cite[Theorem 5.1]{aglermccarthy}}.  The proof is based on \cite{r77} and \cite{k20}. 

In the remaining part of Introduction, definitions and main facts concerning intertwining relations of operators are recalled. 

For a Hilbert space  $\mathcal H$ and a (closed) subspace  $\mathcal M$ of $\mathcal H$,  symbols  $P_{\mathcal M}$  and  $I_{\mathcal H}$ denote the orthogonal projection on $\mathcal M$ and the identity operator on $\mathcal H$, resp. As usually, $\mathcal M^\perp=\mathcal H\ominus\mathcal M$. 
The spectrum and the point spectrum of an operator $T$ are denoted by $\sigma(T)$ and  $\sigma_p(T)$, resp. 

For Hilbert spaces  $\mathcal H$ and $\mathcal K$, the symbol   $\mathcal L(\mathcal H, \mathcal K)$ denotes the space of 
(bounded linear) transformations acting from $\mathcal H$ to $\mathcal K$. Suppose that $T\in\mathcal L(\mathcal H)$, $R\in\mathcal L(\mathcal K)$, $X\in\mathcal L(\mathcal H, \mathcal K)$, and $X$  \emph{intertwines} $T$ and $R$, that is, $XT=RX$. If $X$ is unitary, then $T$ and $R$ 
are called  \emph{unitarily equivalent}, in notation $T\cong R$. If $X$ is invertible, that is, $X^{-1}\in\mathcal L(\mathcal K, \mathcal H)$, 
then $T$ and $R$ are called \emph{similar}, written $T\approx R$.
If $X$ is a \emph{quasiaffinity}, that is, $\ker X=\{0\}$ and $\operatorname{clos}X\mathcal H=\mathcal K$, then
$T$ is called a  \emph{quasiaffine transform} of $R$, in notation $T\prec R$. If $T\prec R$ and 
$R\prec T$, then $T$ and $R$ are called  \emph{quasisimilar}, written $T\sim R$. 
If $\ker X=\{0\}$, then we write 
 $T \buildrel i \over \prec R$, while
if $\operatorname{clos}X\mathcal H=\mathcal K$, we write $T \buildrel d \over \prec R$.
If $\mathcal M\in\operatorname{Lat}T$, then 
$T|_{\mathcal M}\buildrel i \over \prec T\buildrel d \over \prec P_{\mathcal M^\perp}T|_{\mathcal M^\perp}$, 
the last relation is realizes by $ P_{\mathcal M^\perp}$. 

 It follows immediately from the definition that 
if $T \buildrel i \over \prec R$, then $\sigma_p(T)\subset\sigma_p(R)$.  
Also,  
 $T \buildrel d \over \prec R$ if and only if  $R^* \buildrel i \over \prec T^*$. Therefore, if $T \buildrel d \over \prec R$ 
and $\sigma_p(R^*)\neq\emptyset$, then $\sigma_p(T^*)\neq\emptyset$; consequently, $T$ has a nontrivial hyperinvariant subspace.  
Recall that if $T\sim R$ and one of $T$ or $R$ has a  nontrivial hyperinvariant subspace, then so does the other.
 Moreover, if $R\buildrel i \over \prec T\buildrel d\over \prec R$ and  $R$ has a  nontrivial hyperinvariant subspace, then  $T$ has a  nontrivial hyperinvariant subspace, too (a particular case of 
{\cite[Theorem 15]{k20}}, see also the references in \cite{k20}). One of the well-known corollaries is that if $T=N\oplus R$
where $N$ is a normal operator and $R$ is an arbitrary operator then $T$ has a nontrivial hyperinvariant subspace
 (unless $T\neq\lambda I$ for some $\lambda\in \mathbb C$).

\section{Some results for polynomially bounded operators}

\subsection{Definitions and preliminaries}

The symbols $\mathbb D$ and $\mathbb T$ denote the open unit disc
and the unit circle, respectively. The normalized Lebesgue measure on $\mathbb T$ is denoted by $m$. 
The symbol $H^\infty$ denotes the Banach algebra of all bounded analytic functions in $\mathbb D$. 
Set $L^2=L^2(\mathbb T,m)$. 
Set $\chi(z)=z$  and $\mathbf{1}(z)=1$ for $z\in\mathbb D\cup \mathbb T$. 
The simple unilateral shift $S$ and  the simple bilateral shift  $U_{\mathbb T}$ are  
 the operators  of multiplication by $\chi$
 on the Hardy space $H^2$ and on $L^2$, respectively.  
Set $H^2_-=L^2\ominus H^2$. 
By $P_+$ and $P_-$ the orthogonal projections from $L^2$ onto $H^2$ and $H^2_-$  
are denoted, respectively.  
Set $S_*=P_-U_{\mathbb T}|_{H^2_-}$.  Then 
\begin{equation} \label{uu} U_{\mathbb T}=\begin{pmatrix} S & (\cdot, \overline\chi)\mathbf{1}  \\ \mathbb O & S_*\end{pmatrix}
\end{equation}
with respect to the decomposition $L^2=H^2\oplus H^2_-$.

An  operator $T\in\mathcal L(\mathcal H)$ is called  \emph{power bounded} if 
$\sup_{n\geq 0}\|T^n\| < \infty$. It is easy to see that for such operators the space 
\begin{equation}\label{hhtt0}\mathcal H_{T,0}=\{x\in\mathcal H :\ \|T^nx\|\to 0\}
\end{equation}
 is hyperinvariant for $T$ ({\cite[Theorem II.5.4]{sfbk}}). The classes $C_{ab}$ of power bounded operators, where 
 $a$ and $b$ can be $0$, $1$, or a dot,  
are defined as follows. If $\mathcal H_{T,0}=\mathcal H$, then  $T$ is  \emph{of class} $C_{0\cdot}$, while if  $\mathcal H_{T,0}=\{0\}$, then $T$ is 
 \emph{of class} $C_{1\cdot}$. Furthermore,  $T$  is \emph{of class} $C_{\cdot a}$, if $T^*$ is of class  $C_{a\cdot}$,  
 and $T$ is  \emph{of class} $C_{ab}$, if $T$ is of class $C_{a\cdot}$ and of class $C_{\cdot b}$, $a$, $b=0,1$. 

 An operator $T\in\mathcal L(\mathcal H)$ is called  \emph{polynomially bounded} if there exists a constant $C$ such that 
$\|p(T)\|\leq C \sup\{|p(z)|:  |z|\leq 1\} $ for every (analytic) polynomial $p$. For a polynomially bounded operator $T\in\mathcal L(\mathcal H)$ 
there exist $\mathcal H_a$, $\mathcal H_s\in\operatorname{Hlat}T$ such that $ \mathcal H=\mathcal H_a\dotplus\mathcal H_s$,   $T|_{\mathcal H_a}$ is an 
 \emph{absolutely continuous} (\emph{a.c.}) polynomially bounded operator, and $T|_{\mathcal H_s}$ is similar to a 
singular unitary operator. Thus, if $\mathcal H_s\neq\{0\}$, then $T$ has nontrivial  hyperinvariant subspaces. The  definition of a.c. polynomially bounded operators is not recalled here,  because it will be not used. We recall only that \emph{$T$ is an  a.c. polynomially bounded operator if and only if $T$ admits an $H^\infty$-functional calculus}
  \cite{mlak}, {\cite[Theorem 23]{k16}}. 

An operator $T\in\mathcal L(\mathcal H)$  is called 
a  \emph{contraction} if $\|T\|\leq 1$. 
A contraction is polynomially bounded with the constant 
$1$ (von Neumann inequality; 
see, for example, {\cite[Proposition I.8.3]{sfbk}}). Clearly, a polynomially bounded operator is power bounded.

For $\varphi\in H^\infty$ set  $\widetilde{\varphi}(z)=\overline{\varphi(\overline z)}$, $z\in\mathbb D$. 
Clearly, $\widetilde{\varphi}\in H^\infty$. If  $T$ is  an a.c. polynomially bounded operator, then  
$\varphi(T^*)=\widetilde\varphi(T)^*$ (\cite{mlak}, {\cite[Proposition 14]{k16}}). 
An a.c. polynomially bounded operator $T$ is called a \emph{$C_0$-operator},  if there exists $0\not\equiv\varphi\in H^\infty$ such that $\varphi(T)=\mathbb O$, see \cite{bp}. If such $T$ is a contraction, $T$ is called a 
 \emph{$C_0$-contraction}, see \cite{sfbk}. If $T$ is a $C_0$-operator, then $\sigma(T)\cap\mathbb D=\sigma_p(T)$, 
see {\cite[Theorem III.5.1]{sfbk}} and \cite{bp}.  

For a power bounded operator $T\in \mathcal L(\mathcal H)$ the 
\begin{equation*} \text{ \emph{isometric asymptote } }(X_{T,+},V),
\end{equation*}
 where 
$V$ is an isometry and $X_{T,+}$ is the \emph{canonical intertwining mapping} is constructed in \cite{k89}. The definition and construction is not recalled here. We recall only that  $X_{T,+}$ realizes the relation $T\buildrel d \over \prec V$ and 
$\ker X_{T,+}=\mathcal H_{T,0}$, where $\mathcal H_{T,0}$ is defined in \eqref{hhtt0}. If the nonzero isometry $V$ is not a  unitary operator or $\{0\}\neq\ker X_{T,+}\neq\mathcal H$ , then $T$ has a nontrivial  hyperinvariant subspace.  
 The \emph{unitary asymptote}
$(X_T, U)$ of $T$  is the minimal unitary extension of the  isometric asymptote $(X_{T,+},V)$ of $T$. 
More precisely, the unitary operator $U$ is the minimal unitary extension of the isometry $V$, and $X_T$ is equal to
 $X_{T,+}$ considered as a transformation from $\mathcal H$ to the space in which $U$ acts. Conversely, if $(X_T, U)$ is the unitary  asymptote of $T$, then $V=U|_{\operatorname{clos}X_T\mathcal H}$ is the isometry from the isometric asymptote of $T$. 
\emph{The isometry $V$ and the  unitary operator $U$ will also be called the isometric and unitary asymptotes of $T$, respectively.}

The following lemma is a simple corollary of the universality of the isometric asymptote (see \cite{k89}).   

\begin{lemma}\label{lemshift} Let $T$ be a power bounded operator. Then $T\buildrel d \over \prec S$ 
if and only if the isometric asymptote of $T$ is not unitary.
\end{lemma}

\begin{proof} Denote by $(X,V)$ the isometric asymptote of $T$. If $V$ is not unitary, then the Wold decomposition of $V$ implies 
$V \buildrel d \over \prec S$. Since  $T\buildrel d \over \prec V$, we conclude $T \buildrel d \over \prec S$.

Conversely, if there exists a transformation $Y$ with dense range such that $YT=SY$, then by {\cite[Theorem 1]{k89}} 
there exists a transformation $Z$ such that $Y=ZX$ and $ZV=SZ$. If $V$ is unitary, then 
$S|_{\operatorname{clos}\operatorname{ran} Z}$ 
must be of class $C_{\cdot 1}$, which implies  $Z=\mathbb O$, a contradiction.  
\end{proof}

The following proposition is well known.

\begin{proposition}\label{prop1} Let  a power bounded $T$ have the form 
\begin{equation*} T=\begin{pmatrix} T_1 & * \\ \mathbb O & T_0\end{pmatrix},
\end{equation*}
 where $T_1\buildrel d \over \prec S$, and $T_0$ is not of class $C_{\cdot 0}$. Then $T$ has a nontrivial hyperinvariant subspace. 
\end{proposition}
\begin{proof}
The assumption $T_1\buildrel d \over \prec S$ implies that $T_1$ is not of class $C_{0\cdot}$. Consequently, $T$ 
 is not of class $C_{0\cdot}$. The assumption ``$T_0$ is not of class $C_{\cdot 0}$'' implies that $T$ 
 is not of class $C_{\cdot 0}$. Now the conclusion of the proposition follows from 
{\cite[Theorem II.5.4]{sfbk}} or \cite{k89}. \end{proof}

The following lemma is very simple. For the proof, see {\cite[Lemma 1.3]{g19}}.

\begin{lemma}\label{lem13} Suppose  that $T$ and $R$ are a.c. polynomially bounded operators,
$T \buildrel d \over \prec R$ and there exists $\varphi\in H^\infty$ such that $\operatorname{ran} \varphi(R)$ is not dense. 
Then $\operatorname{ran} \varphi(T)$ is not  dense.
\end{lemma}

The following proposition is well known.

\begin{proposition}\label{propaann} Let $A$ be a subnormal power bounded operator of class $C_{0\cdot}$, 
and let $N$ be the minimal normal extension of $A$. 
Then $N$ is a contraction of class $C_{00}$.
\end{proposition}

\begin{proof} Let $r(A)$ denote the spectral radius of $A$. Since $A$ is power bounded, $r(A)\leq 1$. 
Since $A$ is subnormal, $r(A)=\|A\|$ by  {\cite[Corollary II.2.12]{conwaysubnormal}}. 
Since $\|A\|=\|N\|$ by  {\cite[Theorem II.2.11]{conwaysubnormal}}, we conclude that $N$ is a contraction.

Denote by $\mathcal H$ and $\mathcal K$ the space on which $A$ and $N$ acts. 
Then $A=N|_{\mathcal H}$
and \begin{equation}\label{dense}\mathcal K=\vee_{k=0}^\infty N^{*k} \mathcal H.\end{equation}
Let $x\in\mathcal H$, and let $k\geq 0$. Then 
\begin{equation*}\|N^nN^{*k}x\| = \|N^{*k}N^nx\|\leq\|N^nx\|=\|A^nx\|\to 0 \ \text{ when } \ n\to \infty, 
\end{equation*}
because $A$ is of class $C_{0\cdot}$. The last relation, \eqref{dense} and the contractivity of $N$ imply  
that $N$ is of class $C_{0\cdot}$. Consequently,   the  spectral measure of $N$ is concentrated on $\mathbb D$. 
This implies that   $N$ is of class $C_{00}$. 
\end{proof}

\subsection{Main result}

The following proposition is cited from \cite{g19} for reader's convenience.

\begin{proposition}[{\cite[Proposition 3.4]{g19}}] \label{prop34}  Suppose  that  $T_0\in\mathcal L(\mathcal H_0)$ 
is an a.c. polynomially bounded operator, 
$X_0\in\mathcal L(\mathcal H_0, H^2_-)$, $\operatorname{clos}X_0\mathcal H_0=H^2_-$, and $X_0 T_0=S_* X_0$. 
Set
\begin{equation*}\mathbf{T}=\begin{pmatrix} S &  (\cdot, X_0^*\overline\chi)\mathbf{1}  \\ \mathbb O & T_0\end{pmatrix}. 
\end{equation*}
Let $\theta$ be an inner function. Then $\ker \theta(\mathbf{T})^*\neq\{0\}$ if 
 and only if there exists $x_0\in\mathcal H_0$ such that $x_0\notin X_0^* H^2_-$ and
$\theta(T_0)^* x_0 \in  X_0^*  H^2_-$.
\end{proposition}

The following lemma is a corollary of {\cite[Lemma 5.6]{est}}.

\begin{lemma}\label{lemthetanormal} Suppose that  $N$ is a normal contraction of class $C_{00}$, 
$\mathcal H_0\in\operatorname{Lat}N$, and $Y\in\mathcal L(H^2_-,\mathcal H_0)$ is such that
 $Y(S_*)^*=N|_{\mathcal H_0}Y$ 
and $\ker Y=\{0\}$. Then there exist a singular inner function $\theta$ and $x_0\in\mathcal H_0$ such that 
$x_0\not\in YH^2_-$ and $\widetilde\theta(N|_{\mathcal H_0})x_0\in YH^2_-$.
\end{lemma}

\begin{proof} Take $0\not\equiv h_0\in H^2$ such that  $h_0$ has no  singular inner factor.
 Set  $y_0 = Y\overline\chi\overline h_0$. 

Let  $\mu$  be a scalar-valued spectral measure for $N$. Then $\mu$ is  a positive  Borel  measure on $\mathbb D$, 
and 
\begin{equation*}N \cong \oplus_{n=1}^\infty N_{\mu|\Delta_n},
\end{equation*}
where $\Delta_n\subset\mathbb D$ are Borel sets and $N_{\mu|\Delta_n}$ is the operator 
of multiplication by $\chi$ 
on $L^2(\Delta_n,\mu)$. 
(Note that it is not assumed here that $\Delta_n$ are disjoint; moreover, it is possible that 
$\Delta_n=\Delta_k$ for some $n\neq k$. On the other hand, it is possible that $\Delta_n=\emptyset$ for sufficiently large $n$.) We may assume that
\begin{equation*}N = \oplus_{n=1}^\infty N_{\mu|\Delta_n},
\end{equation*}
then $y_0=\oplus_{n=1}^\infty f_n$, where $f_n\in L^2(\Delta_n,\mu)$. 
Set 
\begin{equation*}
\mathrm{d}\alpha(z)=\Big(\sum_{n=1}^\infty|f_n(z)|^2\Big)\mathrm{d}\mu(z).
\end{equation*}
 Then $\alpha(\mathbb D)<\infty$.

Let $0<r_1<\ldots<r_k<r_{k+1}<\ldots<1$ and $r_k\to 1$ when $k\to\infty$. Set 
\begin{equation*}
c_1= \alpha(\{|z|\leq r_1\}) \ \text{ and }  \ c_k= \alpha(\{r_{k-1}<|z|\leq r_k\}), \ \ k\geq 2.
\end{equation*}
Then $\sum_{k=1}^\infty c_k<\infty$. Consequently, there exists a sequence $\{A_k\}_{k=1}^\infty$ such that 
$A_{k+1}>A_k>0$ for every $k\geq 1$,  $A_k\to \infty$ when $k\to\infty$, and $\sum_{k=1}^\infty A_k^2 c_k<\infty$.  
It is easy to construct a function $u\colon(0,1)\to (0,\infty)$ which is continuous, strongly increasing,
 such that  $u(r)\leq A_k$  for $r_{k-1}<r\leq r_k$, $k\geq 2$,  and $u(r)\to\infty$ when $r\to 1$. 
It is easy to see that 
\begin{equation*} \int_{\mathbb D}u(|z|)^2\mathrm{d}\alpha(z)<\infty.\end{equation*}

By {\cite[Lemma 5.6]{est}}, there exists a singular inner function $\vartheta$ such that 
\begin{equation}\label{utheta} \sup_{z\in\mathbb D}\frac{1}{u(|z|)|\vartheta(z)|}=C<\infty.\end{equation}
 
Set $x_0= \oplus_{n=1}^\infty \{f_n/\vartheta\}_{n=1}^\infty$. 
For $0<r<1$ set $\varphi_r(z)=1/\vartheta(rz)$, $z\in\mathbb D$. Then $\varphi_r$ is a function from the disk algebra. 
Since $y_0\in\mathcal H_0$, we have  
 \begin{equation*}\varphi_r(N)y_0=\oplus_{n=1}^\infty \varphi_r f_n\in\mathcal H_0. \end{equation*}
Furthermore, 
\begin{align*}& \Big|\varphi_r(z)-\frac{1}{\vartheta(z)}\Big|^2  \leq\Big(|\varphi_r(z)|+\Big|\frac{1}{\vartheta(z)}\Big|\Big)^2
\leq 2\Big(|\varphi_r(z)|^2+\Big|\frac{1}{\vartheta(z)}\Big|^2\Big)\\& = 
2\Big(\frac{1}{|\vartheta(rz)|^2u(r|z|)^2}\frac{u(r|z|)^2}{u(|z|)^2}+\frac{1}{|\vartheta(z)|^2u(|z|)^2}\Big)u(|z|)^2
\leq 4C^2u(|z|)^2,
\end{align*}
where $C$ is from \eqref{utheta}, because $u(r|z|)\leq u(|z|)$. 
Since $\varphi_r(z)\to1/\vartheta(z)$ when $r\to 1$ for every $z\in\mathbb D$, the Lebesgue convergence theorem implies 
that  
\begin{align*} \|\varphi_r(N)y_0-x_0\|^2&
=\sum_{n=1}^\infty\int_{\mathbb D}\Big|\varphi_rf_n-\frac{1}{\vartheta}f_n\Big|^2\mathrm{d}\mu 
\\&=\int_{\mathbb D}\Big(\sum_{n=1}^\infty|f_n|^2\Big)\Big|\varphi_r-\frac{1}{\vartheta}\Big|^2\mathrm{d}\mu 
\\&= \int_{\mathbb D}\Big|\varphi_r-\frac{1}{\vartheta}\Big|^2\mathrm{d}\alpha \to 0 \text{ when } r\to 1.
\end{align*}
Thus, $x_0\in\mathcal H_0$. It follows from the definition of $x_0$ that $\vartheta(N|_{\mathcal H_0})x_0=y_0$. 

If $x_0=  Y\overline\chi\overline g$ for some  $g\in H^2$, 
then 
\begin{align*} y_0=\vartheta(N|_{\mathcal H_0})x_0=
 \vartheta(N|_{\mathcal H_0})Y\overline\chi\overline g 
=Y\vartheta((S_*)^*)\overline\chi\overline g=
Y \overline{\widetilde\vartheta}\overline\chi\overline g,
\end{align*}
which contradicts with the choice of $y_0$. Thus, $\theta=\widetilde\vartheta$ and $x_0$ satisfy the conclusion of the lemma.
\end{proof}

\begin{theorem}\label{thmmain1}Let an a.c.  polynomially bounded operator $T$ have the form 
\begin{equation} \label{tt120} T=\begin{pmatrix} T_1 & T_2 \\ \mathbb O & T_0\end{pmatrix},
\end{equation}
 where $T_1\buildrel d \over \prec S$, and $T_0^*$ is subnormal. Then $T$ has a nontrivial hyperinvariant subspace. 

Moreover, if $T_0$ is of class $C_{\cdot 0}$, then there exists a singular inner function $\theta$ such that 
 $\operatorname{ran} \theta(T)$ is not  dense.
\end{theorem}

\begin{proof}  If $T_0$ is not of class $C_{\cdot 0}$, then $T$ has a nontrivial hyperinvariant subspace by Proposition \ref{prop1}. 
Thus, it sufficient to consider the case when  $T_0$ is of class $C_{\cdot 0}$. 
Set $A=T_0^*$. 
Then  $A$ satisfies the assumption of  Proposition \ref{propaann}, so $A$ is a subnormal contraction  of class $C_{00}$. 
Consequently, $T_0$ is   of class $C_{00}$. 

Denote by $\mathcal H_1$ and $\mathcal H_0$ the spaces on which $T_1$ and $T_0$ act, respectively, and by $X_1$ 
the transformation which realizes the relation $T_1\buildrel d \over \prec S$. 
For every (analytic) polynomial $p$ set $T_{2,p}=P_{\mathcal H_1}p(T)|_{\mathcal H_0}$.   Set 
\begin{equation}\label{rr} R=\begin{pmatrix} S & X_1T_2 \\ \mathbb O & T_0\end{pmatrix}.\end{equation}
Then $X_1\oplus I_{\mathcal H_0}$ realizes the relation $T\buildrel d \over \prec R$. For every (analytic) polynomial $p$ the equality 
\begin{equation*} p(R)=\begin{pmatrix} p(S) & X_1T_{2,p} \\ \mathbb O & p(T_0)\end{pmatrix}\end{equation*}
is an easy consequence of matrix representations of $p(T)$, $p(R)$, and the form $X_1\oplus I_{\mathcal H_0}$ of 
intertwining transformation. Consequently, $R$ is polynomially bounded. Furthermore, 
 $R$ is absolutely continuous, because $S$ and $T_0$ are 
absolutely continuous (for detailed explanation see {\cite[Lemma 2.2]{g18}}). 
By Lemma \ref{lem13}, it is sufficient to show that there exists a singular inner function $\theta$ such that 
 $\operatorname{ran} \theta(R)$ is not  dense.

If the isometric asymptote of $R$ is not unitary, 
then $R\buildrel d \over \prec S$ by Lemma \ref{lemshift}. Since $\operatorname{ran} \theta(S)$ is not  dense
for every inner function $\theta$, the conclusion of the theorem follows from Lemma \ref{lem13}. 

Now consider the case when  the isometric asymptote $V$ of $R$ is  unitary. By {\cite[Theorem 3]{k89}}, 
 $V\cong U_{\mathbb T}$, because $T_0$ is of class $C_{0\cdot}$. We may assume that $V=U_{\mathbb T}$. Then the canonical intertwining mapping $X$ has the form 
\begin{equation*} X=\begin{pmatrix} I_{H^2} & X_2 \\ \mathbb O & X_0\end{pmatrix}\end{equation*}
with respect to the decompositions \eqref{rr} and \eqref{uu}, where $X_2$ and $X_0$ are appropiate transformations,
and $X_0T_0=S_*X_0$.
The relation 
\begin{equation*} \operatorname{clos}X(H^2\oplus\mathcal H_0)=L^2\end{equation*} 
implies $\operatorname{clos}X_0\mathcal H_0=H^2_-$. 
Let $\mathbf{T}$ be defined as in Proposition \ref{prop34}. 
The relation 
\begin{equation*} 
\begin{pmatrix} I_{H^2} & X_2\\ \mathbb O & I_{\mathcal H_0}\end{pmatrix}R
=\mathbf{T}\begin{pmatrix} I_{H^2} & X_2\\ \mathbb O & I_{\mathcal H_0}\end{pmatrix}
\end{equation*}
means that $R\approx\mathbf{T}$. 

 Denote by $N$ the minimal normal extension of $A$. By Proposition \ref{propaann}, $N$ is of class $C_{00}$. 
 Let a singular inner function  $\theta$ and $x_0\in\mathcal H_0$ be obtained in Lemma \ref{lemthetanormal} applied to 
 $N$, $\mathcal H_0$, and $Y=X_0^*$. 
Note that $T_0^*=A=N|_{\mathcal H_0}$. Therefore, 
\begin{equation*} 
\theta(T_0)^*=\widetilde\theta(T_0^*)=\widetilde\theta(N|_{\mathcal H_0}).
\end{equation*}
Thus, $x_0\notin X_0^* H^2_-$ and
$\theta(T_0)^* x_0 \in  X_0^* H^2_-$. By  Proposition \ref{prop34}, $\ker \theta(\mathbf{T})^*\neq\{0\}$. 
Consequently, $\operatorname{ran} \theta(\mathbf{T})$ is not dense. Since $R\approx\mathbf{T}$,  
$\operatorname{ran} \theta(R)$ is not  dense, too.
\end{proof}

\begin{corollary}\label{coruu} Suppose that in  the assumptions of Theorem \ref{thmmain1} $T_0=T_{00}\oplus U$, where $T_{00}$ 
 is of class $C_{\cdot 0}$ and $U$ is unitary. Then there exists a singular inner function $\theta$ such that 
 $\operatorname{ran} \theta(T)$ is not  dense.
\end{corollary}

\begin{proof} Denote by $\mathcal H$, $\mathcal H_1$, $\mathcal H_{00}$, $\mathcal G$ 
the spaces in which $T$, $T_1$, $T_{00}$, $U$ act, respectively. Then $T$ has the form 
\begin{equation*}T=\begin{pmatrix}T_1 & T_{20} & T_{21} \\ \mathbb O & T_{00} & \mathbb O \\ 
\mathbb O &  \mathbb O & U\end{pmatrix}
\end{equation*}
with respect to the decomposition $\mathcal H=\mathcal H_1\oplus\mathcal H_{00}\oplus\mathcal G$ (where $T_{20}$ and $T_{21}$ are appropriate transformations). 
Set  \begin{equation*}R=T|_{\mathcal H_1\oplus\mathcal H_{00}}=\begin{pmatrix}T_1 & T_{20}  \\ \mathbb O & T_{00}\end{pmatrix}.
\end{equation*}
Then $R$ satisfies the assumptions of Theorem \ref{thmmain1}, because $T_{00}$ is of class $C_{\cdot 0}$ and  $T_{00}^*$ is subnormal. 
Furthermore, $\mathcal G\in\operatorname{Lat}T^*$ and $T^*|_{\mathcal G}=U^*$. Since $U$ is unitary and $T$ is power bounded, 
by \cite{k89b} there exists $\mathcal K\in \operatorname{Lat}T^*$ such that $\mathcal H=\mathcal G\dotplus\mathcal K$. 
(Although the proposition in \cite{k89b} is formulated
for contractions, application of results from \cite{k89} allows to repeat the
proof for power bounded operators.) Consequently, $\mathcal H=\mathcal G^\perp\dotplus\mathcal K^\perp$, and
$\mathcal K^\perp\in\operatorname{Lat}T$. Since  $\mathcal G^\perp=\mathcal H_1\oplus\mathcal H_{00}$, we have
\begin{equation*} 
T=T|_{\mathcal G^\perp}\dotplus T|_{\mathcal K^\perp}=R\dotplus T|_{\mathcal K^\perp}.
\end{equation*}
Therefore,
 $\operatorname{ran} \theta(T)$ is not  dense for every  function $\theta$ such that 
 $\operatorname{ran} \theta(R)$ is not  dense, and such a singular inner function $\theta$
 exists by Theorem \ref{thmmain1} applied to $R$. (Note that 
$T|_{\mathcal K^\perp}\approx U$, where the similarity is realized by $P_{\mathcal G}|_{\mathcal K^\perp}$. Therefore, 
$\theta(T|_{\mathcal K^\perp})\approx\theta(U)$, and, consequently, $\theta(T|_{\mathcal K^\perp})$  is invertible for every inner function $\theta$.) 
\end{proof}

\begin{corollary}\label{cormm} Suppose that in  the assumptions of Theorem \ref{thmmain1} 
there exists a singular inner function $\theta$ such that 
 $\operatorname{ran} \theta(T)$ is not  dense.
Denote by $\mathcal H_1$ the space in which $T_1$ acts, and let $\mathcal M\in\operatorname{Lat}T$ be such that 
$\mathcal H_1\subset\mathcal M$. Then   $\operatorname{clos} \theta(T)\mathcal M\neq\mathcal M$.
\end{corollary}

\begin{proof} Denote by $\mathcal H$ the space in which $T$ acts, set $\mathcal E=\mathcal M^\perp=\mathcal H\ominus\mathcal M$ and $A=T_0^*$. 
By assumption, $A$ is subnormal, and $A$ is an a.c. contraction. Denote by $N$ the minimal normal extension of $A$. 
Then $N$ is an a.c. contraction. Since $\widetilde\theta$ is a singular inner function, 
we have $\widetilde\theta(z)\neq 0$ for every $z\in\mathbb D$ and for $m$-a.e. $z\in\mathbb T$. 
Consequently, $\ker\widetilde\theta(N)=\{0\}$. Therefore, $\ker\widetilde\theta(A)=\{0\}$. Furthermore, 
$\mathcal E\in\operatorname{Lat}A$. 
Therefore, $\ker\widetilde\theta(A|_{\mathcal E})=\{0\}$, too. 
Thus, $\operatorname{clos} P_{\mathcal E}\theta(T_0)\mathcal E=\mathcal E$. If 
 $\operatorname{clos} \theta(T)\mathcal M=\mathcal M$, 
then  $\operatorname{clos} \theta(T)\mathcal H=\mathcal H$, a contradiction with the assumption on $\theta$.
\end{proof}

\begin{remark} Corollary \ref{cormm} can be applied to operators from Corollary \ref{coruu}.
\end{remark}

\begin{remark}\label{rem1} In {\cite[Proposition 2.7]{kimpearcy}} the existence of a nontrivial hyperinvariant subspace 
for an operator of the form 
\begin{equation}\label{ttnn} 
 T=\begin{pmatrix} B & * \\ \mathbb O & N\end{pmatrix},
\end{equation}
where $B=S$ and $N$ is normal, is proved, without the assumption that $T$ is polynomially or power bounded, 
but under some assumptions on $N$,
 especially that $N$ is not cyclic.  
By 
 {\cite[Lemma IV.2.1]{ff}}, for every contractions $T_0$ and $T_1$ such that $T_0$ and $T_1^*$ are not isometries there exists 
a nonzero transformation $T_2$ such that $T$ defined by \eqref{tt120} is a contraction.
Let $\mu$ be a positive Borel measure on $\mathbb D\cup\mathbb T$ such that $\mu(\mathbb D)>0$ and 
$\mu|_{\mathbb T}$ is absolutely continuous with respect to $m$. Let $N_\mu$ be an operator of multiplication by $\chi$ on $L^2(\mu)$. Then $N_\mu$ is a cyclic normal operator, and $N_\mu$ is an a.c. contraction. 
 Applying {\cite[Lemma IV.2.1]{ff}} to $T_1=S$ 
and $T_0=N_\mu$,
 examples of contractions of the form \eqref{ttnn} with $B=S$ can be obtained which do
not satisfy the assumption of  {\cite[Proposition 2.7]{kimpearcy}} and satisfy the assumptions of 
Corollary  \ref{coruu}. 
Using {\cite[Proposition 3.1]{g19}}
 and the fact that $S\prec R$ for any cyclic a.c. contraction $R$ 
which is not a $C_0$-contraction (see {\cite[Introduction]{tak}} and references therein) and taking a positive Borel measure $\mu$
 on $\mathbb D$ without atoms, 
 it is easy to construct an operator $T$  which is similar to   a.c. contraction of class $C_{10}$ 
with the isometric asymptote $U_{\mathbb T}$ such that  $T$ satisfies  the assumptions of 
Theorem \ref{thmmain1} with $T_0=N_\mu$, where $N_\mu$  is a cyclic normal contraction of class $C_{00}$. For more details, see the next subsection.  

In {\cite[Corollary 3.7]{jungkopearcy}} the existence of a nontrivial hyperinvariant subspace 
for an operator of the form \eqref{ttnn} is proved under the  assumption that $N$ is normal and there exists a nonzero transformation $X$
 such that $XN=BX$. If $B$ is pure subnormal, in particular if $B=S$, and $XN=BX$,  then $X=\mathbb O$. 
 Indeed, the equality $XN=BX$ with a nonzero $X$ implies 
\begin{equation*}(N^*|_{(\ker X)^\perp})^*=P_{(\ker X)^\perp} N|_{(\ker X)^\perp}\prec
 B|_{\operatorname{clos}\operatorname{ran}X},\end{equation*}
and $N^*|_{(\ker X)^\perp}$ is subnormal. A contradiction is obtained by  {\cite[Proposition II.10.6]{conwaysubnormal}}. 
\end{remark}

\subsection{Examples and additional propositions}

For $\nu\in\mathbb N\cup\{\infty\}$ denote by  $H^2_\nu$, $L^2_\nu$, $(H^2_-)_\nu$   the orthogonal sum of $\nu$ copies of $H^2$, $L^2$, $H^2_-$, respectively.
 By $P_+$ 
 the orthogonal projection from $L^2_\nu$ onto $H^2_\nu$ 
is denoted (it depends on $\nu$, but it will not be mentioned in notation). 
By $S_\nu$, $S_{*, \nu}$,  and $U_{\mathbb T, \nu}$
 the orthogonal sum of $\nu$ copies of $S$, $S_*$, and $U_{\mathbb T}$ are denoted,  respectively. 
Set $K_\nu=P_+ U_{\mathbb T,\nu}|_{(H^2_-)_\nu}$. Then 
\begin{equation} \label{uunu} U_{\mathbb T,\nu}=\begin{pmatrix} S_\nu & K_\nu  \\ \mathbb O & S_{*,\nu}\end{pmatrix}
\end{equation}
with respect to the decomposition $L^2_\nu=H^2_\nu\oplus (H^2_-)_\nu$. 
The following proposition is a simple generalization of   {\cite[Proposition 3.1]{g19}}
and allows to find the isometric asymptote of the constructed operator.

\begin{proposition}\label{prop31}  Suppose  that $\nu\in\mathbb N\cup\{\infty\}$, $T_0\in\mathcal L(\mathcal H_0)$ 
is a contraction of class $C_{00}$, 
$X_0\in\mathcal L(\mathcal H_0, (H^2_-)_\nu)$,  and $X_0 T_0=S_{*,\nu} X_0$. 
Set
\begin{equation*} T=\begin{pmatrix} S_\nu & K_\nu X_0 \\ \mathbb O & T_0\end{pmatrix}. 
\end{equation*}
Then $T$ is similar to a contraction of class $C_{\cdot 0}$, and 
$((I_{H^2_\nu}\oplus X_0), U_{\mathbb T, \nu})$ is the unitary asymptote of $T$. 
Consequently, $T$ is of class $C_{10}$ if and only if $\ker X_0=\{0\}$, 
and  $U_{\mathbb T,\nu}$ is the isometric asymptote of $T$ if and only if 
 $\operatorname{clos}X_0\mathcal H_0=(H^2_-)_\nu$. 
\end{proposition}
\begin{proof}
It is easy to see from the definition of $T$ and \eqref{uunu} that 
 \begin{equation*} (I_{H^2_\nu}\oplus X_0)T=U_{\mathbb T, \nu}(I_{H^2_\nu}\oplus X_0). \end{equation*} 
Therefore, for every (analytic) polynomial $p$ we have  
\begin{equation*} (I_{H^2_\nu}\oplus X_0)p(T)=p(U_{\mathbb T, \nu})(I_{H^2_\nu}\oplus X_0).  \end{equation*} 
This implies \begin{equation*} 
p(T)=\begin{pmatrix} p(S_\nu) & P_+p(U_{\mathbb T, \nu})|_{(H^2_-)_\nu} X_0 \\ \mathbb O & p(T_0)\end{pmatrix}. 
\end{equation*}
Consequently, $T$ is polynomialy bounded. Since $S_\nu$ and $T_0$ are of class $C_{\cdot 0}$,
 it follows from {\cite[Theorem 3]{k89}} that $T$ is of class $C_{\cdot 0}$, too. 
The statements on unitary and isometric asymptotes are proved as in the proof of  {\cite[Proposition 3.1]{g19}}. 

The conclusion on similarity to a contraction follows from {\cite[Corollary 4.2]{cassier}}.
\end{proof}

 \begin{example} Let $\nu\in\mathbb N\cup\{\infty\}$, and let $\mu$ be  a positive  Borel  measure on $\mathbb D$ without atoms. 
Then there exist disjoint Borel sets $\{\Delta_n\}_{n=1}^\nu$ such that 
$\mathbb D=\cup_{n=1}^\nu\Delta_n$ and $\mu(\Delta_n)>0$ for every $n$. 
Let $N_{\mu|\Delta_n}$ be the operator 
of multiplication by the independent variable on $L^2(\Delta_n,\mu)$. Then $N_{\mu|\Delta_n}$ is a contraction 
of class $C_{00}$, and  $N_{\mu|\Delta_n}$ is not of class $C_0$, because $\mu$ has no atoms. 
By {\cite[Theorem V.14.21]{conwaysubnormal}},  $N_{\mu|\Delta_n}$ is cyclic. Consequently, $S\prec  N_{\mu|\Delta_n}$ 
for every $n$  (see {\cite[Introduction]{tak}} and references therein).

Let $N_\mu$ be the operator 
of multiplication by the independent variable on $L^2(\mu)$. Then  
\begin{equation*}N_\mu\cong \oplus_{n=1}^\nu N_{\mu|\Delta_n},
\end{equation*}
because $\Delta_n$ are disjoint. Thus, $S_\nu\prec N_\mu$. Consequently, $N_\mu^*\prec S_{*,\nu}$. 
Denote by $X_0$ a transformation which realizes the last relation and apply Proposition \ref{prop31} with $\nu$ 
and $T_0=N_\mu^*$. 
The obtained operator $T$ is of class $C_{10}$, and its isometric asymptote is $U_{\mathbb T, \nu}$, because $X_0$ is a quasiaffinity. 
  $T$ satisfies the assumption of Theorem \ref{thmmain1}. Note that  $N_\mu^*$ is cyclic by {\cite[Theorem V.14.21]{conwaysubnormal}}. 
\end{example}

\begin{example} Let $A$ be a subnormal contraction of class  $C_{00}$, $\sigma_p(A)=\emptyset$, and $\dim\ker A^*=\infty$. 
Examples of such subnormal operators are restrictions of Bergman shifts on appropiate invariant subspaces, see \cite{hrs}  and 
\cite{bhv}. 
By {\cite[Theorem 1]{tak}}, $S_\infty\prec A$. Set $T_0=A^*$. Then $T_0\prec S_{*,\infty}$. 
  Denote by $X_0$ a transformation which realizes the last relation and apply Proposition \ref{prop31} with $\nu=\infty$ and $T_0$. 
The obtained operator $T$ is of class $C_{10}$, and its isometric asymptote is $U_{\mathbb T,\infty}$, 
because $X_0$ is a quasiaffinity.  $T$ satisfies the assumption of Theorem \ref{thmmain1}.
\end{example}

\begin{example} Let  $\nu\in\mathbb N$, and let $A$ be a subnormal contraction of class  $C_{00}$
 with $\sigma_p(A)=\emptyset$. By {\cite[Theorem 2]{tak}}, 
$S_\nu\buildrel i\over\prec A$. 
Denote by $Y$ the transformation which realizes the last relation.  
Set $T_0=(A|_{\operatorname{clos}YH^2_\nu})^*$. Then $T_0\prec S_{*,\nu}$. 
As in two previous examples, apply Proposition \ref{prop31} with $\nu$ and $T_0$. 
The obtained operator $T$ is of class $C_{10}$,  its isometric asymptote is $U_{\mathbb T,\nu}$, and 
 $T$ satisfies the assumption of Theorem \ref{thmmain1}.
\end{example}

The following propositions give some examples of operators $T$ such that $T\buildrel d \over \prec S$. 
They have the form 
\begin{equation}\label{tt100}T=\begin{pmatrix}T_1 & * \\ \mathbb O & T_{00}\end{pmatrix} ,
\end{equation}
where $T_1 \buildrel d \over \prec S_\nu$ for some $\nu\in\mathbb N\cup\{\infty\}$ and $T_{00}\in C_{0\cdot}$.

\begin{proposition} \label{propnu0}Suppose that $\nu_0\in\mathbb N$, $\nu_1\in\mathbb N\cup\{\infty\}$, $\nu_0<\nu_1$, 
and a power bounded operator $T$ has the form \eqref{tt100}, where $T_1 \buildrel d \over \prec S_{\nu_1}$ and 
$S_{*,\nu_0}\ \buildrel d \over \prec T_{00}$. 
 Then $T\buildrel d \over \prec S$. 
\end{proposition}

\begin{proof} As in the beginning of the proof of Theorem \ref{thmmain1}, there exists a power bounded operator $R$ such that 
 $T\buildrel d \over \prec R$ and 
\begin{equation}\label{rr01} R=\begin{pmatrix}S_{\nu_1} & * \\ \mathbb O & T_{00}\end{pmatrix}.
\end{equation}
(Considerations related to polynomially boundedness in the construction of $R$ in Theorem \ref{thmmain1} can be replaced by 
considerations related to power boundedness.)  It is sufficient to prove the proposition for $R$. 

Denote by $\mathcal H_{00}$ the space on which $T_{00}$ acts. Assume that the isometric asymptote $V$ of $R$ is  unitary. By {\cite[Theorem 3]{k89}}, 
 $V\cong U_{\mathbb T,\nu_1}$, because $T_{00}$ is of class $C_{0\cdot}$. We may assume that $V=U_{\mathbb T,\nu_1}$. Then the canonical intertwining mapping $X$ has the form 
\begin{equation*} X=\begin{pmatrix} I_{H^2_{\nu_1}} & * \\ \mathbb O & X_0\end{pmatrix}\end{equation*}
with respect to the decompositions \eqref{rr01} and \eqref{uunu}, where  $X_0$ is a transformation such that  
$X_0T_{00}=S_{*,\nu_1}X_0$.
The relation 
\begin{equation*} \operatorname{clos}X(H^2_{\nu_1}\oplus\mathcal H_{00})=L^2_{\nu_1}\end{equation*} 
implies $\operatorname{clos}X_0\mathcal H_{00}=(H^2_-)_{\nu_1}$. Thus, 
$S_{*,\nu_0} \buildrel d \over \prec T_{00}\buildrel d \over \prec S_{*,\nu_1}$. 
Therefore 
\begin{equation*} S_{\nu_1}\cong (S_{*,\nu_1})^*\buildrel i \over \prec (S_{*,\nu_0})^*\cong S_{\nu_0},
\end{equation*}
a contradiction with the assumption $\nu_0<\nu_1$, see {\cite[Theorem 5/6]{sznf}}. Consequently, the isometric asymptote of $R$ is not unitary. 
By Lemma \ref{lemshift}, $R\buildrel d \over \prec S$. 
\end{proof}

\begin{proposition}\label{prop00}
Let  a polynomially bounded operator $T$ have the form \eqref{tt100}, 
 where $T_1\buildrel d \over \prec S$,  and $T_{00}$ is a $C_0$-operator.  Then $T\buildrel d \over \prec S$.
\end{proposition}

\begin{proof} As in the beginning of the proof of Theorem \ref{thmmain1}, there exists a polynomially  bounded operator $R$ such that 
 $T\buildrel d \over \prec R$ and 
\begin{equation*}
 R=\begin{pmatrix}S & * \\ \mathbb O & T_{00}\end{pmatrix}.
\end{equation*}
By {\cite[Lemma 5.2]{g25}}, the isometric asymptote of $R$ is $S$. Consequently,  $R\buildrel d \over \prec S$.
\end{proof}

Recall that  {\cite[Lemma IV.2.1]{ff}} allows to construct contractions satisfying assumptions of Propositions \ref{propnu0} and \ref{prop00} (not simultaneously) with nonzero transformations denoted by $*$ in \eqref{tt100}.

\section{Intertwining with some operators from both sides}

It is known during long time that if an operator $T$ has a normal orthogonal summand and is not a scalar multiple of the identity operator, then $T$ has a nontrivial hyperinvariant subspace. By {\cite[Theorem 5.1]{aglermccarthy}}, if $T$ is intertwined from both sides (by nonzero transformations) with some normal operators, then $T$ has a nontrivial invariant subspace. In this section a generalization of these statements is proved. 

Recall the definitions from \cite{r77}. 
An operator $A\in\mathcal L(\mathcal H)$ is called \emph{dominant}, if 
for every $\lambda\in\mathbb C$ there exists a constant $M_\lambda$ such that 
 $\|(A-\lambda)^*x\|\leq M_\lambda^{1/2}\|(A-\lambda)x\|$ for every  $x\in\mathcal H$.  Let $M>0$. An operator $A\in\mathcal L(\mathcal H)$ is called \emph{$M$-hyponormal}, 
if $\|(A-\lambda)^*x\|\leq M^{1/2} \|(A-\lambda)x\|$ for every  $\lambda\in\mathbb C$ and every $x\in\mathcal H$. 
 An $M$-hyponormal operator is dominant. Let $\mathcal M\in\operatorname{Lat} A$. It is easy to see that if $A$ is $M$-hyponormal (respectively, dominant), 
then $A|_{\mathcal M}$ is $M$-hyponormal (respectively, dominant). 
If $M=1$, then $M$-hyponormal operators are  \emph{hyponormal} (see {\cite[Proposition II.4.4(b)]{conwaysubnormal}}).  By {\cite[Proposition II.4.2]{conwaysubnormal}}, 
 every subnormal operator is hyponormal. 

Note that the definition of $M$-hyponormal operator is different from another definitions, where other letters instead of 
 $M$ are used. The references on many papers where such definitions are given and the correspondent properties of operators are considered are not given here. 
  
 \begin{theorem}\label{thmmain2} Suppose that $A$ is an $M$-hyponormal operator, $B$ is a  dominant operator, 
 $W_1$ and $W_2$ are nonzero transformations, and $T$ is an operator such that $W_1A^*=TW_1$ and $W_2T=BW_2$. 
If $T\neq\lambda I$ for any $\lambda \in\mathbb C$, then $T$ has a nontivial hyperinvariant subspace. 
\end{theorem}

\begin{proof} Denote by $\mathcal G$, $\mathcal H$,  and $\mathcal K$ the spaces on which $A$, $T$ and $B$ act, respectively. 
Set \begin{equation*}\widetilde{\mathcal K}=\operatorname{clos}W_2\mathcal H,
 \ \ \widetilde{\mathcal G}=\mathcal G\ominus\ker W_1, 
\ \ \widetilde W_1=W_1|_{\widetilde {\mathcal G}}, 
\end{equation*} 
and define $\widetilde W_2$ as $W_2$ acting from $\mathcal H$ to $\widetilde{\mathcal K}$. Then 
\begin{equation*} \begin{gathered}\widetilde{\mathcal K}\in\operatorname{Lat} B,  \ \ \widetilde{\mathcal G}\in\operatorname{Lat} A,   \\
\ \ \widetilde W_1(A|_{\widetilde{\mathcal G}})^*=T\widetilde W_1,
\ \ \widetilde W_2 T= B|_{\widetilde{\mathcal K}}\widetilde W_2, 
\ \ \ker \widetilde W_1=\{0\},   \ \operatorname{clos}\widetilde W_2\mathcal H = \widetilde{\mathcal K}.\end{gathered}
\end{equation*} 
$A|_{\widetilde{\mathcal G}}$ is $M$-hyponormal and $B|_{\widetilde{\mathcal K}}$ is dominant. 
Therefore, we may assume that 
\begin{equation*} \ker  W_1=\{0\} \ \text{ and }  \ \operatorname{clos} W_2\mathcal H = \mathcal K.
\end{equation*} 

Set $\mathcal M_1=\operatorname{clos} W_1\mathcal G$, $\mathcal M_2=\ker W_2$, $Z_2=W_2|_{\mathcal M_2^\perp}$, 
and define $Z_1$ as $W_1$  acting from $\mathcal G$ to  $\mathcal M_1$. Then $\mathcal M_1$,
 $\mathcal M_2\in\operatorname{Lat}T$, $\mathcal M_1\neq\{0\}$, and $\mathcal M_2\neq\mathcal H$. 
For every $R\in\{T\}'$ set $R_0=P_{\mathcal M_2^\perp}R|_{\mathcal M_1}$. 
Then  \begin{equation*}Z_2 R_0 Z_1 A^* = Z_2 R_0 T|_{\mathcal M_1}  Z_1=
Z_2 P_{\mathcal M_2^\perp}T|_{\mathcal M_2^\perp} R_0 Z_1 = B Z_2 R_0 Z_1.  
\end{equation*} 
Consider two cases. \emph{First case}: for every $R\in\{T\}'$ we have $Z_2 R_0 Z_1=\mathbb O$. 
Since $Z_1$ and $Z_2$ are quasiaffinities, the last equality means that $R\mathcal M_1\subset\mathcal M_2$ 
for every $R\in\{T\}'$. Consequently, 
 \begin{equation*} \mathcal N=\vee_{R\in\{T\}'}Rx\in\operatorname{Hlat}T \ \text{ and }
 \ \{0\}\neq\mathcal N\subset\mathcal M_2\neq\mathcal H
\end{equation*} 
for every $0\neq x\in\mathcal M_1$.  

\emph{Second case}: there exists $R\in\{T\}'$ such that $Z_2 R_0 Z_1\neq\mathbb O$. Set 
\begin{equation*}\mathcal K_1=\operatorname{clos}Z_2 R_0 Z_1\mathcal G, \ \ \mathcal G_1=\mathcal G\ominus\ker Z_2 R_0 Z_1,
\end{equation*}
and let $X=Z_2 R_0 Z_1|_{\mathcal G_1}$ be considered as a transformation from $\mathcal G_1$ to $\mathcal K_1$. 
Then $X$ is a quasiaffinity which realizes the relation 
\begin{equation*} (A|_{\mathcal G_1})^*\prec B|_{\mathcal K_1}.
\end{equation*}
Since $A|_{\mathcal G_1}$ is $M$-hyponormal and $B|_{\mathcal K_1}$ is dominant, 
it follows from the last relation and {\cite[Theorem 3(a)]{r77}} 
that $(A|_{\mathcal G_1})^*$ and $B|_{\mathcal K_1}$ are normal,
 and  $(A|_{\mathcal G_1})^*\cong B|_{\mathcal K_1}$. By  {\cite[Theorem 4]{r77}}, 
$\mathcal G_1$ and $\mathcal K_1$  are reducing subspaces for $A$ and $B$, respectively. 
Consequently, $A\buildrel d\over \prec A|_{\mathcal G_1}$ and $B\buildrel d\over \prec B|_{\mathcal K_1}$.
Set $N=B|_{\mathcal K_1}$. Then
\begin{equation*} N\buildrel i\over \prec A^*\buildrel i\over \prec T\buildrel d\over \prec B\buildrel d\over \prec N. 
\end{equation*} 
If $N$ has a nontrivial hyperinvariant subspace, then $T$ has a nontrivial hyperinvariant subspace by 
{\cite[Theorem 15]{k20}}. If $N$ has no nontrivial hyperinvariant subspace, it means that $N=\lambda I$ 
for some $\lambda\in\mathbb C$, because $N$ is normal. 
Then the relation $\lambda I\buildrel i\over\prec T$ implies that $\lambda\in\sigma_p(T)$. If $T\neq \lambda I$, then  
 $T$ has a nontrivial hyperinvariant subspace.
\end{proof}
 
Note that if $A$ and $B$ in the assumption of Theorem  \ref{thmmain2} are subnormal,
 then {\cite[Theorem 3(a)]{r77}} in the proof of Theorem \ref{thmmain2} can be replaced by {\cite[Proposition II.10.6]{conwaysubnormal}}.


\end{document}